\documentclass[a4paper,11pt,reqno]{amsart}
\pdfoutput=1

\usepackage{amsthm,amsmath,amssymb}
\usepackage{graphicx}
\usepackage[colorlinks=true,citecolor=black,linkcolor=black,urlcolor=blue]{hyperref}

\newcommand{\arxiv}[1]{\href{http://arxiv.org/abs/#1}{\texttt{arXiv:#1}}}

\usepackage{enumerate}
\usepackage{stmaryrd}
\usepackage{url}

\graphicspath{{figures/}}

\newcommand{\N}{\mathbb{N}}

\newcommand{\Q}{\mathbb{Q}}
\newcommand{\Id}{\textnormal{Id}}
\newcommand{\ev}{\textnormal{ev}}
\newcommand{\ASM}{\textnormal{ASM}}
\newcommand{\NC}{\textnormal{NC}}
\newcommand{\D}{\mathbf{D}}
\renewcommand{\S}{\mathbf{S}}

\theoremstyle{plain}
\newtheorem{thm}{Theorem}
\newtheorem{lem}[thm]{Lemma}

\newtheorem{prop}[thm]{Proposition}

\theoremstyle{definition}
\newtheorem{defi}[thm]{Definition}
\newtheorem{ex}[thm]{Example}

\theoremstyle{remark}
\newtheorem{rem}[thm]{Remark}

\title[FPLs: polynomiality and nested arches]{Fully packed loop configurations: polynomiality and nested arches}
\author{Florian Aigner}
\address{Florian Aigner, Universit\"at Wien, Fakult\"at f\"ur Mathematik, Oskar-Morgenstern-Platz~1, 1090 Wien, Austria}
\email{florian.aigner@univie.ac.at}
\thanks{Supported by the Austrian Science Foundation FWF, START grant Y463.}

\begin{document}
\maketitle

\begin{abstract}
This article proves a conjecture by Zuber about the enumeration of fully packed loops (FPLs). The conjecture states that the number of FPLs whose link pattern consists of two noncrossing matchings which are separated by $m$ nested arches is a polynomial function in $m$ of certain degree and with certain leading coefficient.
Contrary to the approach of Caselli, Krattenthaler, Lass and Nadeau (who proved a partial result) we make use of the theory of wheel polynomials developed by Di Francesco, Fonseca and Zinn-Justin.
We present a new basis for the vector space of wheel polynomials and a polynomiality theorem in a more general setting. This allows us to finish the proof of Zubers conjecture.

\end{abstract}

\section{Introduction}

Alternating sign matrices (ASMs) are combinatorial objects with many different faces. They were introduced by Robbins and Rumsey in the 1980s and arose from generalizing the determinant. Together with Mills, they \cite{ASMs} conjectured a closed formula for the enumeration of ASMs of given size, first proven by Zeilberger \cite{Zeilberger}. Using a second guise of ASMs, the six vertex model,
Kuperberg \cite{Kuperberg} could find a different proof for their enumeration. A more detailed account on the history of the ASM Theorem can be found in \cite{ASM_Book}. \\

A third way of looking at ASMs are fully packed loops (FPLs). We obtain by using the FPL description a natural refined counting $A_\pi$ of ASMs by means of noncrossing matchings. Razumov and Stroganov \cite{Razumov-Stroganov-Conj} conjecturally connected  FPLs to the $O(1)$ loop model, a model in statistical physics. Proven by Cantini and Sportiello \cite{Razumov-Stroganov-Proof}, this connection allows a description of $(A_\pi)_{\pi\in \NC_n}$ as an eigenvector of the Hamiltonian of the $O(1)$ loop model, where $\NC_n$ is the set of noncrossing matchings of size $n$. Assuming the (at that point unproven) Razumov-Stroganov conjecture to be true, Zuber \cite{Zuber} formulated nine conjectures about the numbers $A_\pi$. In this paper we finish the proof of the following conjecture.

\begin{thm}[{\cite[Conjecture $7$]{Zuber}}]
\label{thm: main thm}
 For noncrossing matchings $\pi_1 \in \NC_{n_1},$ $\pi_2 \in \NC_{n_2}$ and an integer $m$, the number of FPLs with link pattern $(\pi_1)_m \pi_2$ is a polynomial in $m$ of degree $|\lambda(\pi_1)|+|\lambda(\pi_2)|$ with leading coefficient $\frac{f_{\lambda(\pi_1)}f_{\lambda(\pi_2)}}{|\lambda(\pi_1)|! |\lambda(\pi_2)|!}$, where $f_\lambda$ denotes the number of standard Young tableaux of shape $\lambda$.
\end{thm}

Caselli, Krattenthaler, Lass and Nadeau \cite{on_the_number_of_FPL} proved this for empty $\pi_2$ and showed that $A_{(\pi_1)_m\pi_2}$ is a polynomial for large values of $m$ with correct degree and leading coefficient. In this paper we prove that the number $A_{(\pi_1)_m\pi_2}$ is a polynomial function in $m$, which is achieved without relying on the work of \cite{on_the_number_of_FPL}, and hence finish together with the results of \cite{on_the_number_of_FPL} the proof of Theorem \ref{thm: main thm}.\\

We conclude the introduction by sketching the theory on which the proof of Theorem \ref{thm: main thm} relies and giving an overview of this paper. In the next section we introduce the combinatorial objects and their notions.

As mentioned before the Razumov-Stroganov-Cantini-Sportiello Theorem \ref{thm: Razumov-Stroganov} states that $(A_\pi)_{\pi\in \NC_n}$ is up to multiplication by a constant the unique eigenvector to the eigenvalue $1$ of the Hamiltonian of the homogeneous $O(1)$ loop model.
In Section 3 we present that in a special case solutions of the quantum Knizhnik-Zamolodchikov (qKZ) equations lie in the eigenspace to the eigenvalue $1$ of the Hamiltonian of the inhomogeneous $O(1)$ loop model. Di Francesco and Zinn-Justin \cite{Around_the_RS_conj} could characterise the components of these solutions in a different way, namely as wheel polynomials. The specialisation of the inhomogeneous to the homogeneous $O(1)$ loop model means for wheel polynomials performing the evaluation $z_1=\ldots=z_{2n}=1$. Summarising, for every $\pi \in \NC_n$ there exists an element $\Psi_\pi$ of the vector space $W_n[z]$ of wheel polynomials such that $A_\pi=\Psi_\pi(1,\ldots,1)$.
\begin{align*}
&\textnormal{FPLs} \mathrel{\mathop{\longleftrightarrow}^{\textnormal{RSCS -- Thm}}} \textnormal{hom }O(1)
&\mathrel{\mathop{\longleftarrow}^{\textnormal{specialisation}}} &\textnormal{inhom }O(1) \mathrel{\mathop{\longleftrightarrow}^{\textnormal{Di F. -- Z.\,J.\,}}} & W_n[z]\\
&A_\pi =\Psi_\pi(1,\ldots,1) &\mathrel{\mathop{\longleftarrow}^{\textnormal{evaluation}}} & &\Psi_\pi
\end{align*}
We introduce a new family of wheel polynomials $D_{\pi_1,\pi_2}$ such that every $\Psi_{\rho^{n_2}(\pi_1\pi_2)}$ is a linear combination of $D_{\sigma_1,\sigma_2}$ where $\rho$ is the rotation acting on noncrossing matchings and for $i=1,2$ the Young diagram $\lambda(\sigma_i)$ is included in the Young diagram $\lambda(\pi_i)$.

The advantage of the wheel polynomials $D_{\pi_1,\pi_2}$ over $\Psi_{\pi_1\pi_2}$ becomes clear in Section 4. We prove in Lemma \ref{lem: evaluation of P} in a more general setting that $D_{\pi_1,\pi_2}(1,\ldots,1)$ is a polynomial function with degree at most $|\lambda(\pi_1)|+|\lambda(\pi_2)|$.
This lemma applied in our situation and using the rotational invariance $A_\pi = A_{\rho(\pi)}$ imply the polynomiality in Theorem \ref{thm: main thm}.\\

An extended abstract of this work was published in the Proceedings of FPSAC 2016 \cite{Aigner}.

\section{Definitions}
This section should be understood as a handbook of the combinatorial objects involved in this paper.

\subsection{Noncrossing matchings and Young diagrams}

 A \emph{noncrossing matching} of size $n$ consists of $2n$ points on a line labelled from left to right with the numbers $1, \ldots, 2n$ together with $n$ pairwise noncrossing arches above the line such that every point is endpoint of exactly one arch. An example can be found in Figure \ref{fig: nc matching 1}.
 Denote for two noncrossing matchings $\sigma, \pi$ by $\sigma\pi$ their concatenation. For an integer $n$ we define $(\pi)_n$ as the noncrossing matching $\pi$ surrounded by $n$ nested arches, see Figure \ref{fig: nc matching 2}. Define $\NC_n$ as the set of noncrossing matchings of size $2n$. It is easy to see that $|\NC_n|=C_n=\frac{1}{n+1}\binom{2n}{n}$, where $C_n$ is the $n$-th Catalan number.\\

\begin{figure}
 \centering
 \includegraphics[width=0.5\textwidth]{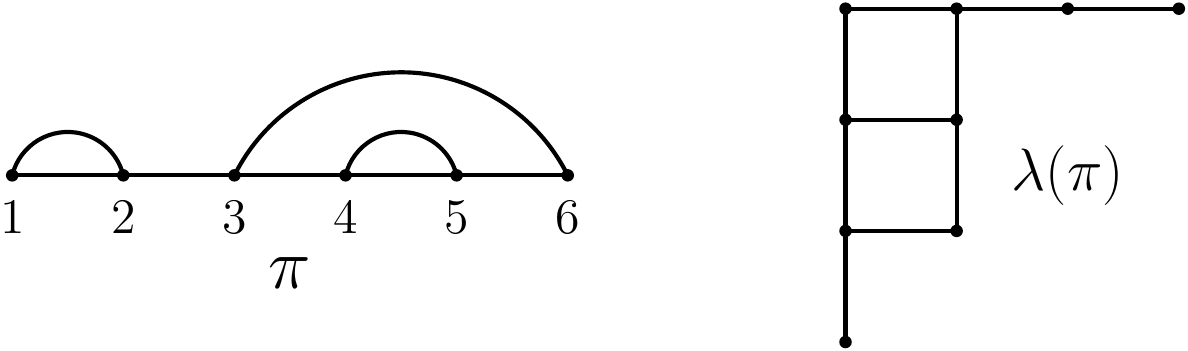}
 \caption{\label{fig: nc matching 1} A noncrossing matching $\pi$ of size $3$ and its corresponding Young diagram $\lambda(\pi)$.}
\end{figure}
\begin{figure}
 \centering
 \includegraphics[width=0.75\textwidth]{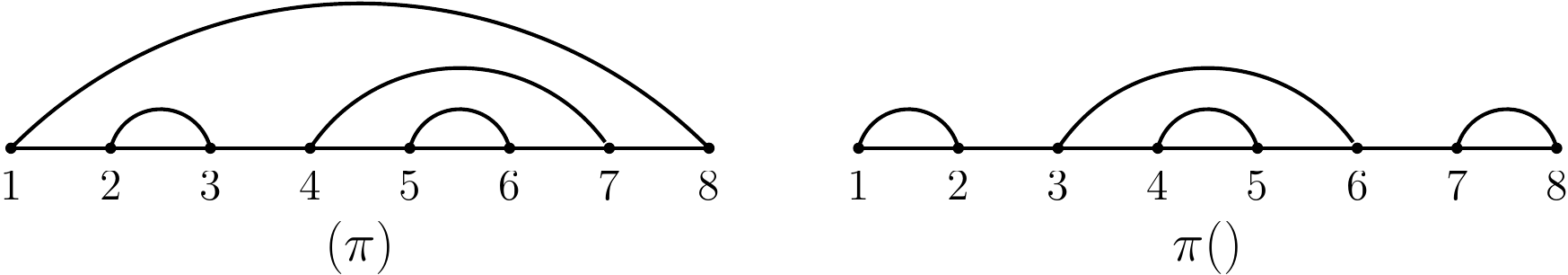}
 \caption{\label{fig: nc matching 2} The noncrossing matchings $(\pi)$ and $\pi()$ where $\pi$ is the noncrossing matching of Figure \ref{fig: nc matching 1}.}
\end{figure}

A \emph{Young diagram} is a finite collection of boxes, arranged in left-justified rows and weakly decreasing row-length from top to bottom. We can think of a Young diagram $\lambda$  as a partition $\lambda=(\lambda_1,\ldots, \lambda_l)$, where $\lambda_i$ is the number of boxes in the $i$-th row from top. Noncrossing matchings of size $n$ are in bijection to Young diagrams for which the $i$-th row from top has at most $n-i$ boxes for $1 \leq i \leq n$. For a noncrossing matching $\pi$ its corresponding Young diagram $\lambda(\pi)$ is given by the area enclosed between two paths with same start- and endpoint. The first path consists of $n$ consecutive north-steps followed by $n$ consecutive east-steps. We construct the second path by reading the numbers from left to right and drawing a north-step if the number labels a left-endpoint of an arch and an east-step otherwise. An example of a noncrossing matching and its corresponding Young diagram is given in Figure \ref{fig: nc matching 1}. For a given noncrossing matching $\pi$ and a positive integer $k$ the Young diagrams $\lambda(\pi)$ and $\lambda((\pi)_k)$ are the same. To be able to distinguish between them we will always draw the first path of the above algorithm in the pictures of $\lambda(\pi)$.\\

We define a partial order on the set $\NC_n$ of noncrossing matchings via $\sigma < \pi$ iff the Young diagram $\lambda(\sigma)$ is contained in the Young diagram $\lambda(\pi)$. For $2 \leq j \leq 2n-2$ we write $\sigma \nearrow_j \pi$ if $\lambda(\pi)$ is obtained by adding a box to $\lambda(\sigma)$ on the $j$-th diagonal, where the diagonals are labelled as in Figure \ref{fig: refined order}. This labelling of the diagonals is the second reason for drawing the consecutive north and east steps in the pictures of the Young diagrams.

\begin{figure}
 \centering
 \includegraphics[width=0.75\textwidth]{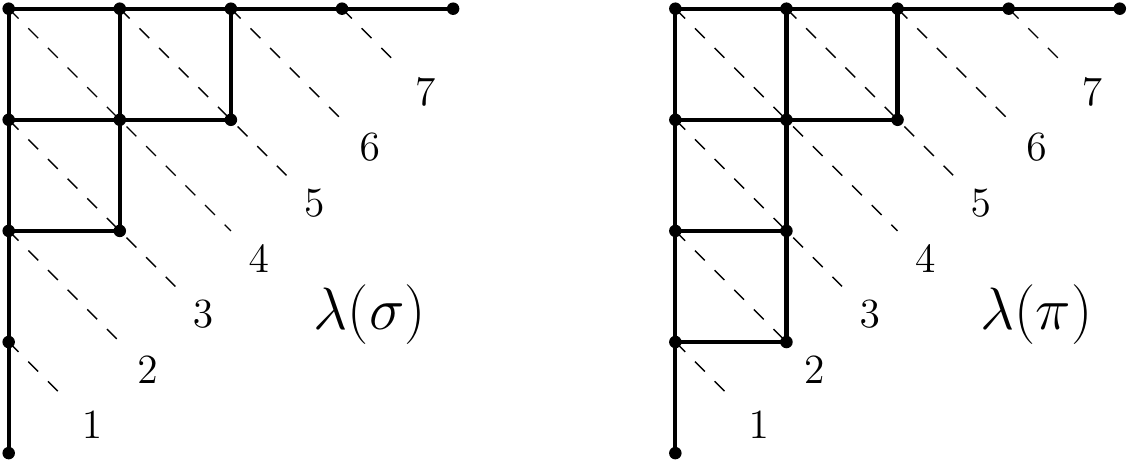}
 \caption{\label{fig: refined order} The matchings $\sigma,\pi$ satisfy $\sigma \nearrow_2 \pi$.}
\end{figure}

\subsection{The Temperley-Lieb Operators}
\label{sec: temp-lieb}
We define first the \emph{rotation} $\rho: \NC_n \rightarrow \NC_n$. Two numbers $i$ and $j$ are connected in $\rho(\pi)$ for $\pi \in \NC_n$ iff $i-1$ and $j-1$ are connected in $\pi$, where we identify $2n+1$ with $1$. The \emph{Temperley-Lieb operator} $e_j$ for $1 \leq j \leq 2n$ is a map from noncrossing matchings of size $n$ to themselves. For a given $\pi \in \NC_n$ the noncrossing matching $e_j(\pi)$ is obtained by deleting the arches which are incident to the points $j,j+1$ and adding an arch between $j,j+1$ and an arch between the points former connected to $j$ and $j+1$. Thereby we identify $2n+1$ with $1$.
There exists also a graphical representation of the Temperley-Lieb operators. Applying $e_j$ on a noncrossing matching $\pi$ is done by attaching the diagram of $e_j$,  depicted in Figure \ref{fig: graphical temperley},  at the bottom of the diagram of $\pi$ and simplifying the paths to arches. An example for this is given in Figure \ref{fig: graphical temperley calculation}.\\

\begin{figure}
 \centering
 \includegraphics[width=0.6\textwidth]{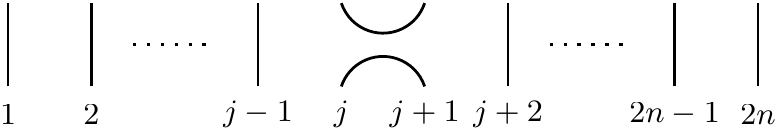}
 \caption{\label{fig: graphical temperley} The graphical representation of $e_j$}
\end{figure}

\begin{figure}
 \centering
 \includegraphics[width=\textwidth]{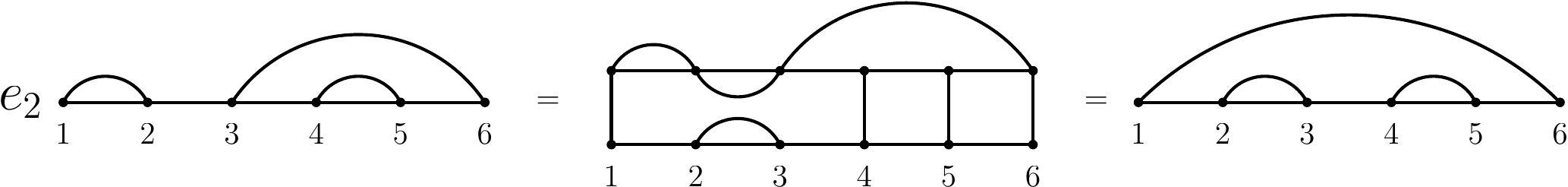}
 \caption{\label{fig: graphical temperley calculation} Calculating $e_j(\pi)$ graphically with $\pi$ from the previous example.}
\end{figure}

Since noncrossing matchings of size $n$ are in bijection with Young diagrams whose $i$-th row from the top has at most $n-i$ boxes, we can define $e_j$ also for such Young diagrams via $e_j(\lambda(\pi)):=\lambda(e_j(\pi))$. For $1 \leq j \leq 2n-1$ the action of $e_j$ on Young diagrams is depicted in Figure  \ref{fig: temperley for Young1}. The operator $e_{2n}$ maps a Young diagram to itself iff the $i$-th row has less than $n-i$ boxes for all $1 \leq i \leq n-1$. Otherwise the Young diagram corresponds to a noncrossing matching of the form $(\alpha)\beta(\gamma)$, where $\alpha,\beta,\gamma$ are noncrossing matchings of smaller size. In this case $e_{2n}$ maps this Young diagram to the one corresponding to the noncrossing matching $(\alpha(\beta)\gamma)$, as depicted in Figure \ref{fig: temperley for Young2}. The next lemma is an easy consequence of the above observations.

\begin{figure}
 \begin{minipage}{0.5\textwidth}
  \centering
  \includegraphics[width=0.9\textwidth]{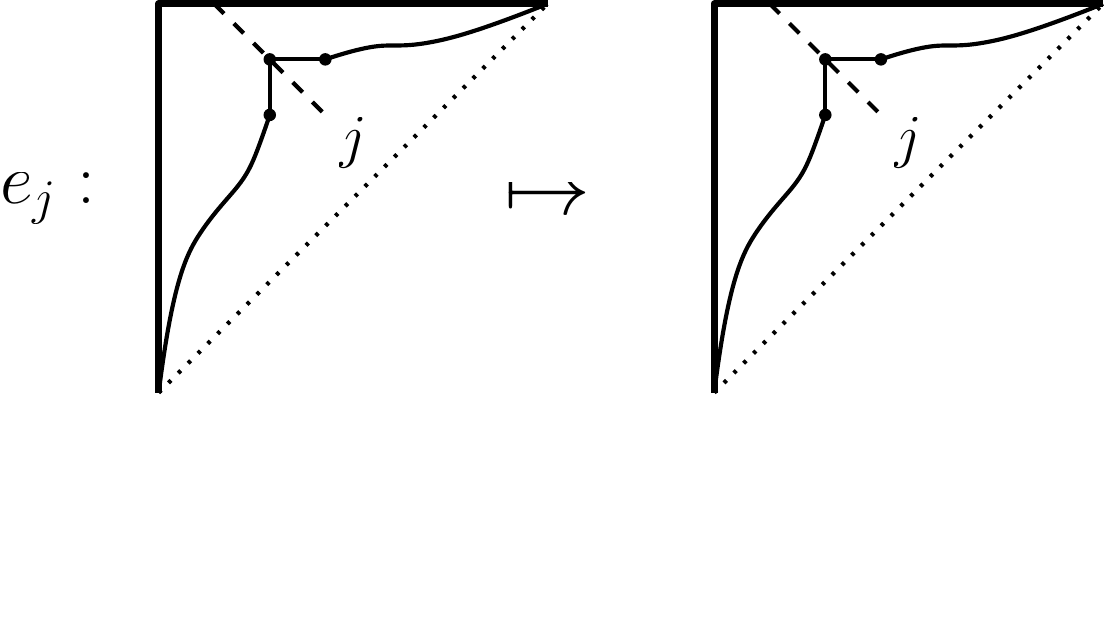}
  \includegraphics[width=0.9\textwidth]{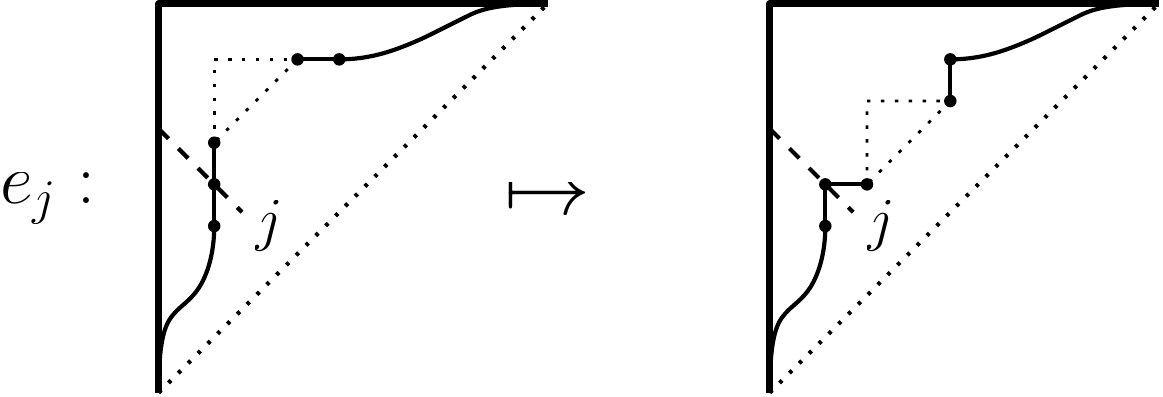}
 \end{minipage}%
 \begin{minipage}{0.5\textwidth}
  \centering
  \includegraphics[width=0.9\textwidth]{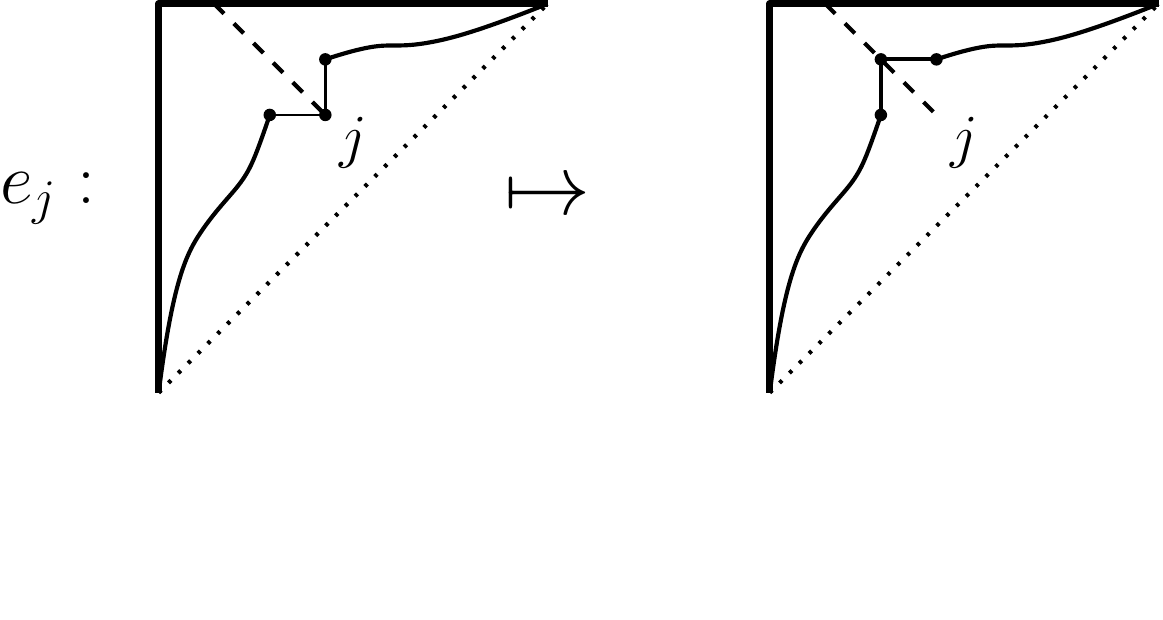}
  \includegraphics[width=0.9\textwidth]{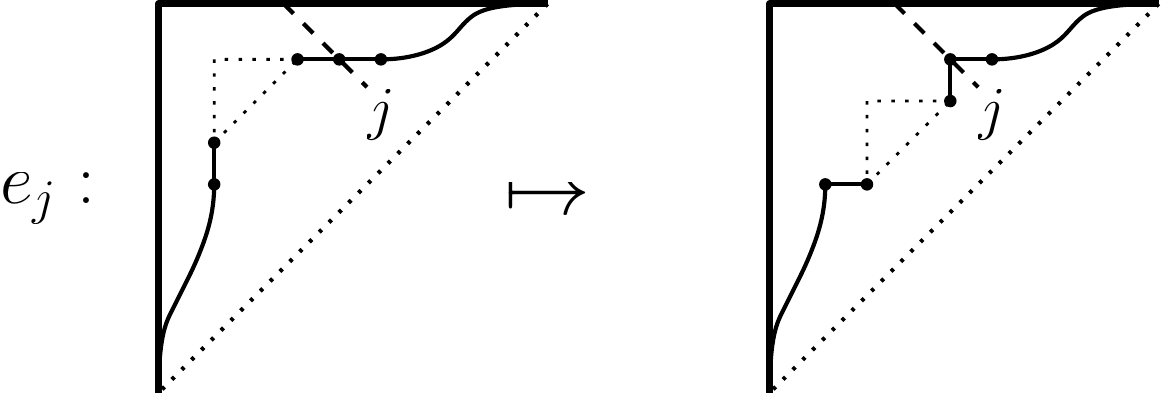}
 \end{minipage}
 \caption{\label{fig: temperley for Young1} The action of $e_j$ for $1 \leq j \leq 2n-1$ on Young diagrams corresponding to noncrossing matchings of size $n$.}
\end{figure}

\begin{figure}
 \centering
 \includegraphics[width=0.6\textwidth]{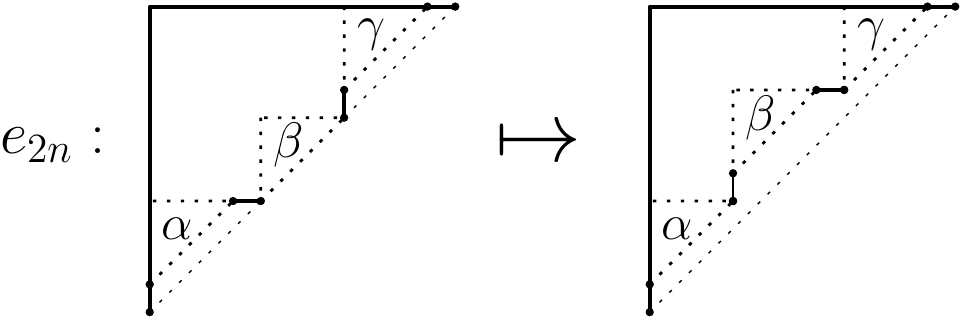}
 \caption{\label{fig: temperley for Young2} The action of $e_{2n}$ on Young diagrams corresponding to noncrossing matchings of size $n$ of the form $(\alpha)\beta(\gamma)$, where $\alpha,\beta,\gamma$ are noncrossing matchings.}
\end{figure}

\begin{lem}
 \label{lem: temperley lieb operator for Young}
 \begin{enumerate}
 \item  For a noncrossing matching $\pi$ of size $n$ and $2 \leq j \leq 2n-2$, the preimage $e_j^{-1}(\pi)$ is a subset of $\{\sigma |  \pi \nearrow_j \sigma\} \cup \{\sigma | \sigma \leq \pi\}$.
 \item  Let $\alpha \in \NC_n$, $\beta, \gamma \in \NC_{n^\prime}$ be noncrossing matchings such that there exists $2 \leq i \leq 2n^\prime-2$ with $\beta \nearrow_i \gamma$. Then the preimage $ e_{2n+i}^{-1} (\alpha\beta)$ is given by
 \[
  e_{2n+i}^{-1} (\alpha\beta)=\{\alpha\sigma|\sigma \in e_i^{-1}(\beta)\}.
 \]
\end{enumerate}
\end{lem}
\begin{proof}
\begin{enumerate}
 \item If $\pi$ has no arch between $j$ and $j+1$, then $e_j^{-1}(\pi)=\emptyset$. Figure \ref{fig: temperley for Young2} displays the action of $e_j$ on Young diagrams and implies the statement if $\pi$ has an arch between $j$ and $j+1$.
 \item Let $\sigma \in e_{2n+i}^{-1}(\alpha\beta)$ and denote by $x,y$ the labels which are connected in $\sigma$ to $2n+i$ or $2n+i+1$ respectively. By definition of $e_{2n+i}$ the noncrossing matchings $\alpha\beta$ and $\sigma$ differ only in the arches between $2n+i,2n+i+1,x,y$. The existence of an $\gamma$ with $\beta \nearrow_i \gamma$ means there exists an arch in $\beta$ with left-endpoint before $i$ and right-endpoint after $i$, hence surrounding $2n+i$ and $2n+i+1$. Therefore $x$ and $y$ must be surrounded by this arch or they are the labels of the points connected by this arch. In both cases $x,y \geq 2n$ which implies $\sigma$ can be written as $\alpha \sigma^\prime$ with $e_i(\sigma^\prime)=\beta$. \qedhere
\end{enumerate}
\end{proof}

The \emph{Temperley-Lieb algebra} with parameter $\tau=-(q+q^{-1})$ of size $2n$ is generated by the Temperley-Lieb operators $e_i$ with $1 \leq i \leq 2n$ over $\mathbb{C}$. The elements $e_i,e_j$ satisfy for all $1 \leq i,j \leq 2n$ the following relations
\begin{align*}
 e_i^2&=\tau e_i,\\
 e_ie_j&=e_je_i \qquad \textit{if } 2 \leq |(i-j)| \leq 2n-2,\\
 e_ie_{i \pm 1}e_i&=e_i.
\end{align*}
Throughout this paper we interpret $e_i v$ for some vector $v \in \{f | f:\NC_n \rightarrow V \}$ and a vector space $V$ always as the action of an element of the Temperley-Lieb algebra on the vector $v$, where the Temperley-Lieb operators act as permutations, i.\,e., $e_i((v_\pi)_{\pi \in \NC_n})=(v_{e_i(\pi)})_{\pi \in \NC_n}$.

\subsection{Fully packed loop configurations}

A \emph{fully packed loop configuration} (or short \emph{FPL}) $F$ of size $n$ is a subgraph of the $n \times n$ grid with $n$ external edges on every side with the following two properties.
\begin{enumerate}
 \item All vertices of the $n\times n$ grid have degree $2$ in $F$.
 \item $F$ contains every other external edge, beginning with the topmost at the left side.
 \end{enumerate}

 An FPL consists of pairwise disjoint paths and loops. Every path connects two external edges. We number the external edges in an FPL counter-clockwise with $1$ up to $2n$, see Figure \ref{fig: FPL}. This allows us to assign to every FPL $F$ a noncrossing matching $\pi(F)$, where $i$ and $j$ are connected by an arch in $\pi(F)$ if they are connected in $F$. We call $\pi(F)$ the \emph{link pattern} of $F$ and write $A_\pi$ for the number of FPLs $F$ with link pattern $\pi(F)=\pi$.\\

\begin{figure}
 \centering
 \includegraphics[width=0.75\textwidth]{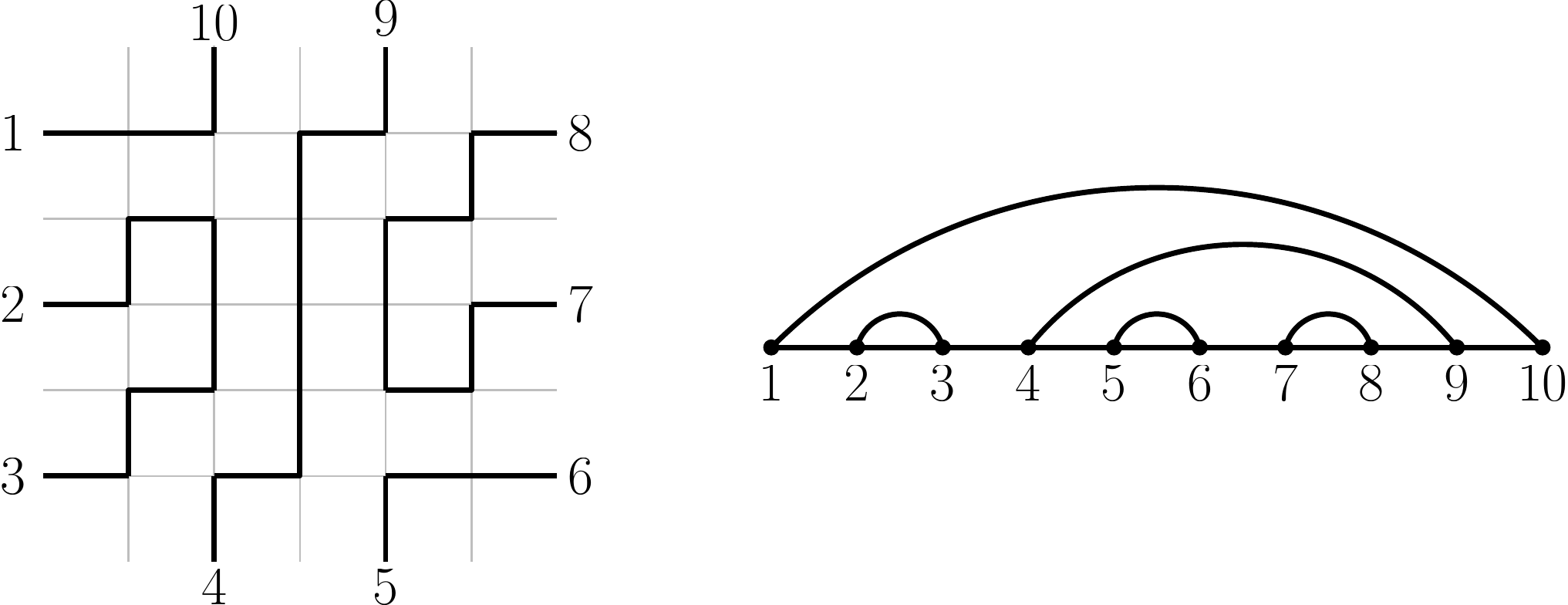}
 \caption{\label{fig: FPL} An example of an FPL of size $5$ and its link pattern.}
\end{figure}

It is well known that FPLs and alternating sign matrices (ASMs) are in bijection. The number of FPLs of size $n$, denoted by $\ASM(n)$, is hence given by the ASM-Theorem \cite{ASMs,Zeilberger}
 \[
\ASM(n)=\prod_{i=0}^{n-1} \frac{(3i+1)!}{(n+i)!}.
\]

%


\subsection{The (in-)homogeneous $O(\tau)$ loop model}
A configuration of the \emph{inhomogeneous $O(\tau)$ loop model} of size $n$ is a tiling of $[0,2n]\times [0,\infty)$ with plaquettes of side length $1$ depicted in Figure \ref{fig: plaquettes}. To obtain a cylinder we identify the half-lines $\{(0,t),t\geq 0\}$ and $\{(2n,t), t\geq 0\}$. In the following we assume that the cylindrical loop percolations are filled randomly with the two plaquettes, where the probability to place the first plaquette of Figure \ref{fig: plaquettes} in column $i$ is $p_i$ with $0 < p_i < 1$ for all $1 \leq i \leq 2n$.
If the probability does not depend on the column, i.\,e., $p_1= \ldots =p_{2n}$, we call it the \emph{homogeneous $O(\tau)$ loop model}. We parametrise the probabilities $p_i=\frac{q z_i-q^{-1}t}{q t-q^{-1}z_i}$ and $\tau=q +q^{-1}$. The two plaquettes in Figure \ref{fig: plaquettes} are interpreted to consist of two paths. By concatenating the paths of a plaquette with the paths of the neighbouring plaquettes, we see that a cylindrical loop percolation consists of noncrossing paths.

\begin{figure}
 \centering
 \includegraphics[width=0.25\textwidth]{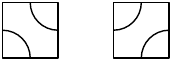}
 \caption{\label{fig: plaquettes} The two different plaquettes.}
\end{figure}

\begin{figure}
 \centering
 \includegraphics[width=0.5\textwidth]{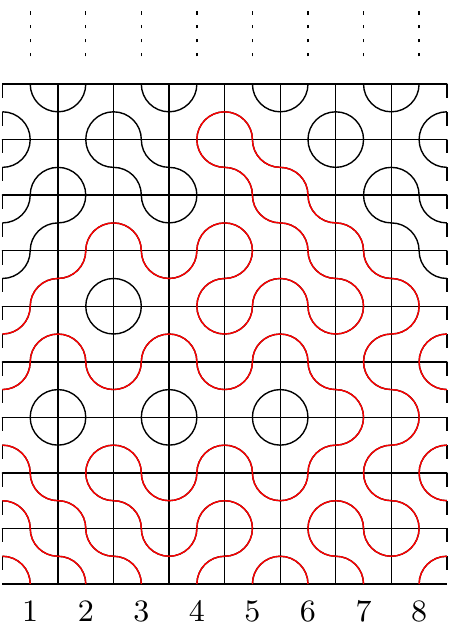}
 \caption{\label{fig: cylindrical loop percolation} The beginning of a cylindrical loop percolation, where the paths starting and ending at the bottom are drawn in red.}
\end{figure}

\begin{lem}
\label{lem: paths are finite}
With probability $1$ all paths in a random cylindrical loop percolation are finite.
\end{lem}

A proof for the homogeneous case can be found in \cite[Lemma 1.6]{Romik}, the inhomogeneous case can be proven analogously. For a configuration $C$ of the $O(\tau)$ loop model, we label the points $(i-\frac{1}{2},0)$ with $i$ for $1 \leq i \leq 2n$. We define the \emph{connectivity pattern} $\pi(C)$ as the noncrossing matching connecting $i$ and $j$ by an arch iff $i$ and $j$ are connected by paths in $C$. By the above lemma $\pi(C)$ is well defined for almost all cylindrical loop percolations $C$. For $\pi \in \NC_n$ denote by $\hat\Psi_\pi(t;z_1, \ldots,z_{2n})$ the probability that a configuration $C$ has the connectivity pattern $\pi$ and write $\hat\Psi_n(t;z_1, \ldots,z_{2n})=(\hat\Psi_\pi(t;z_1, \ldots,z_{2n}))_{\pi \in \NC_n}$.\\

\begin{figure}
 \centering
 \includegraphics[width=0.75\textwidth]{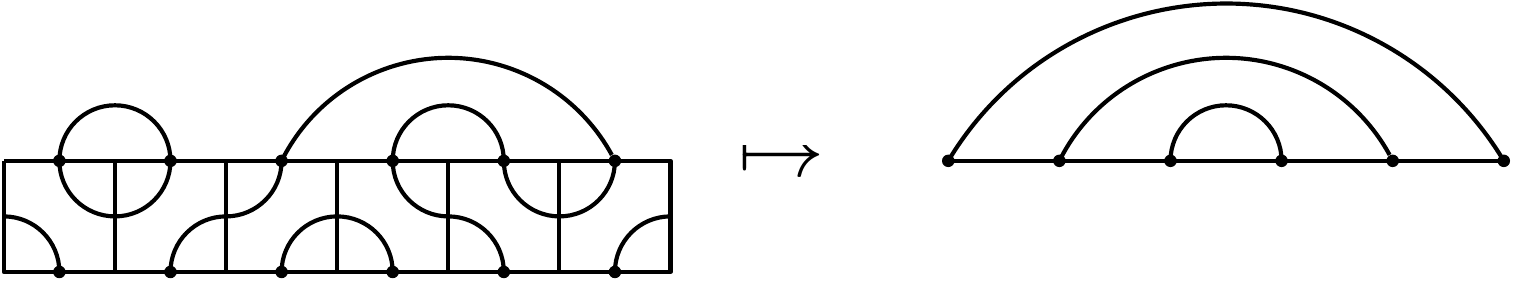}
 \caption{\label{fig: nc matching Markov} An example for a state transition starting with the noncrossing matching $\pi$ of Figure \ref{fig: nc matching 1}. The transition probability is $p_1 (1-p_2)(1-p_3)p_4p_5(1-p_6)$.}
\end{figure}

We define a Markov chain on the set $\NC_n$ of noncrossing matchings of size $n$. The transitions are given by putting $2n$ plaquettes below a noncrossing matching and simplify the paths to obtain a new noncrossing matching. An example is given in Figure \ref{fig: nc matching Markov}. The probability of one transition is given by the product of the probabilities of placing the plaquettes, where placing the first plaquette of Figure \ref{fig: plaquettes} at the $i$-th position is $p_i$ as before. We denote by  $T_n(t;z_1,\ldots, z_{2n})$ the transition matrix of this Markov chain. By the Perron-Frobenius Theorem the matrix $T_n(t;z_1,\ldots, z_{2n})$ has $1$ as an eigenvalue and the stationary distribution of the Markov chain is up to scaling the unique eigenvector with associated eigenvalue $1$. Every configuration $C$ of the inhomogeneous $O(\tau)$ loop model can be obtained uniquely by pushing all the plaquettes of a configuration $C^\prime$ one row up and filling the empty bottom row with plaquettes. Therefore the vector $\hat\Psi_n(t;z_1, \ldots,z_{2n})$ is the stationary distribution of this Markov chain and hence satisfies

\begin{equation}
 \label{eq: inhomogeneous Markov formula}
 T_n(t;z_1,\ldots, z_{2n})\hat\Psi_n(t;z_1, \ldots,z_{2n})=\hat\Psi_n(t;z_1, \ldots,z_{2n}).
\end{equation}

We define the Hamiltonian as the linear map $\mathcal{H}_n:=\sum_{j=1}^{2n}e_j$, where $e_j$ is interpreted as an element of the Temperley-Lieb algebra.

\begin{thm}
 \label{thm: Hamiltonian}
The stationary distribution $\hat\Psi_n(t)=\hat\Psi_n(t;1, \ldots,1)$ satisfies for $\tau=1$
 \begin{align}
  \label{eq: hamiltonian}
  \mathcal{H}_n(\hat\Psi_n(t))=2n \hat\Psi_n(t).
 \end{align}
Further $\hat\Psi_n(t)$ is independent of $t$ and uniquely determined by \eqref{eq: hamiltonian}.
\end{thm}
A proof of this theorem can be found for example in \cite[Appendix B]{Romik}, however note that the matrix $H_n$ defined there is given by $2n \cdot \Id-\mathcal{H}_n$.\\

\noindent The following theorem was conjectured by Razumov and Stroganov in \cite{Razumov-Stroganov-Conj} and later proven by Cantini and Sportiello in \cite{Razumov-Stroganov-Proof}. It creates a connection between fully packed loop configurations and the stationary distribution of the homogeneous $O(1)$ loop model.

\begin{thm}[Razumov-Stroganov-Cantini-Sportiello Theorem]
\label{thm: Razumov-Stroganov}
 Let $n\in \N$, set $q=e^{\frac{2\pi i}{3}}$  and $\hat\Psi_\pi=\hat\Psi_\pi(-q;1,\ldots,1)$. For all $\pi \in \NC_n$ holds 
 \[
  \hat\Psi_\pi=\frac{A_\pi}{\ASM(n)}.
 \]
\end{thm}

\section{The vector space $W_n[z]$}

\subsection{The quantum Knizhnik-Zamolodchikov equations}

In order to introduce the quantum Knizhnik-Zamolodchikov equations (qKZ-equations), we need to define first the $R$-matrix  and the operator $S_i$
\begin{align*}
 \check{R}_i(u)&=\frac{(q u-q^{-1})\Id+(u-1)e_i}{q-q^{-1}u},\\
  S_i(z_1,\ldots, z_{2n}) &= \prod_{k=1}^{i-1}\check{R}_{i-k}\left(\frac{z_{i-k}}{q^6z_{i}}\right) \rho \prod_{k=1}^{2n-i}\check{R}_{2n-k}\left(\frac{z_{2n-k+1}}{z_{i}}\right),
\end{align*}
for $1 \leq i \leq 2n$, where $e_i$ is understood as an element of the Temperley-Lieb algebra and $\rho$ is the rotation as defined in section \ref{sec: temp-lieb}.
%
Denote by $\Psi_n = (\Psi_\pi)_{\pi \in \NC_n}$ a function in $z_1, \ldots, z_{2n},q$. The level $1$ \emph{qKZ-equations} are a system of $2n$ equations
\begin{equation}
 \label{eq: qKZ}
 S_i(z_1,\ldots, z_{2n})\Psi_n(t;z_1,\ldots, z_{2n})=\Psi_n(t;z_1,\ldots,q^6z_i,\ldots, z_{2n}),
\end{equation}
with $1 \leq i \leq 2n$. In the following we need the $2n+1$ equations

\begin{subequations}
\label{eq: general eq}
\begin{align}
\label{eq: Rmatrix applied on Psi}
 \check{R}_i\left(\frac{z_{i+1}}{z_i}\right)\Psi_n(t;z_1,\ldots, z_{2n})&=\Psi_n(t;z_1,\ldots,z_{i+1},z_i,\ldots, z_{2n}),\\
 \label{eq: rho applied on Psi}
 \rho^{-1}\Psi_n(t;z_1,\ldots, z_{2n})&=\Psi_n(t;z_2,\ldots, z_{2n},q^6z_1),
\end{align}
\end{subequations}
where $1 \leq i \leq 2n$ in \eqref{eq: Rmatrix applied on Psi}.

\begin{prop}[{\cite[section 4.1 and 4.3]{Zinn-Justin-Habil}}]
 \label{prop: qKz and O(tau)}
 \begin{enumerate}
  \item The system of equations \eqref{eq: Rmatrix applied on Psi} and \eqref{eq: rho applied on Psi} imply the system of  equations \eqref{eq: qKZ}.
  \item For $q=e^{\frac{2 \pi i}{3}}$ and hence $\tau=1$, it holds $S_i(z_1,\ldots, z_{2n})=T_n(z_i;z_1,\ldots,z_{2n})$. By using Lagrange interpolation one can show that \eqref{eq: qKZ} imply \eqref{eq: inhomogeneous Markov formula}. Since the solutions of \eqref{eq: inhomogeneous Markov formula} form a one dimensional vector space, the same is true for solutions of the system of equations \eqref{eq: qKZ} for $q=e^{\frac{2 \pi i}{3}}$.
 \end{enumerate}
\end{prop}

\subsection{Wheel polynomials}

It turns out \cite[Theorem 4]{Around_the_RS_conj} that for $q=e^{\frac{2 \pi i}{3}}$ the components $\hat\Psi_\pi(t;z_1,\ldots, z_{2n})$ of the stationary distribution of the inhomogeneous $O(1)$ loop model are up to a common factor homogeneous polynomials in $z_1,\ldots, z_{2n}$ of degree $n(n-1)$ which are independent of $t$. In this section we characterise these homogeneous polynomials. In fact we characterise homogeneous solutions of degree $n(n-1)$ of \eqref{eq: Rmatrix applied on Psi} and \eqref{eq: rho applied on Psi} which are by Proposition \ref{prop: qKz and O(tau)} for $q=e^{\frac{2\pi i}{3}}$ also solutions of \eqref{eq: inhomogeneous Markov formula} . The results presented here can be found in \cite{Around_the_RS_conj,doubly,ground_state, Zinn-Justin-Habil,QKZ-Equation} and \cite{Romik}.

\begin{defi}
 Let $n$ be a positive integer and $q$ a variable. A homogeneous polynomial $p \in \Q(q)[z_1,\ldots, z_{2n}]$ of degree $n(n-1)$ is called \emph{wheel polynomial} of order $n$ if it satisfies the \emph{wheel condition}:
 \begin{equation}
  p(z_1,\ldots, z_{2n})|_{q^4z_i=q^2z_j=z_k}=0,
 \end{equation}
 for all triples $1 \leq i<j<k \leq 2n$.
 Denote by $W_n[z]$ the $\Q(q)$-vector space of wheel polynomials of order $n$.
\end{defi}

\begin{thm}[{\cite[Section 4.2]{doubly}}]
 \label{thm: dual basis of Wn}
 The dual space $W_n[z]^*$ of  $W_n[z]$ is given by
 \[
  W_n[z]^* = \bigoplus_{\pi \in \NC_n}\Q(q) \ev_\pi,
 \]
 where $\ev_\pi$ is defined as $\ev_\pi: p(z_1,\ldots, z_{2n}) \mapsto p(q^{\epsilon_1(\pi)},\ldots,q^{\epsilon_{2n}(\pi)})$ with $\epsilon_i(\pi)=-1$ iff an arch of $\pi$ has a left-endpoint labelled with $i$ and $\epsilon_i(\pi)=1$ otherwise.
\end{thm}

Define the linear maps $\S_k,\D_k: \Q(q)[z_1,\ldots, z_{2n}] \rightarrow \Q(q)[z_1,\ldots, z_{2n}]$ for $1 \leq k \leq 2n$ as
\begin{align}
 \label{eq: defi of S_k}
 &\S_k: \quad f(z_1,\ldots, z_{2n}) \; \mapsto \; f(z_1,\ldots, z_{k+1},z_{k},\ldots, z_{2n}),\\
 \label{eq: defi of D_k}
 &\D_k: \quad f \; \mapsto \;   \frac{q z_k-q^{-1}z_{k+1}}{z_{k+1}-z_k}(\S_k(f)-f).
\end{align}
By setting $\D_{k+2n}:=\D_k$ we extend the definition of $\D_k$ to all integers $k$.\\
The operators $\D_k$ are introduced as an abbreviation for $(q z_k-q^{-1}z_{k+1})\delta_k$, where $\delta_k=\frac{1}{z_{k+1}-z_k}(\S_k-\Id)$ has been used before , e.\,g., in \cite{Zinn-Justin-Habil}.
One can verify easily the following Lemma.

\begin{lem}
\label{lem: S,D operator}
\begin{enumerate}
 \item  The space $W_n[z]$ of all wheel polynomials of order $n$ is closed under the action of $\D_k$ for
  $1 \leq k \leq 2n-1$. If $q=e^{\frac{2 \pi i}{3}}$ the vector space $W_n[z]$ is also closed under $\D_{2n}$.
  \item For all $1 \leq k \leq 2n$ and all polynomials $f,g \in \mathbb{Q}(q)[z_1,\ldots, z_{2n}]$ one has
  \begin{align}
  \label{eq: product rule}
   \D_k(f g) =\D_k(f)\S_k(g) +f \D_k(g).
  \end{align}
\end{enumerate}
\end{lem}

The following theorem describes a very important $\Q(q)$-basis of $W_n[z]$.

\begin{thm}[{\cite[Section 4.2]{Zinn-Justin-Habil}}]
\label{thm: psi basis for wheel poly}
Set
 \begin{align}
 \label{eq: Psi_null}
 \Psi_{()_n}(z_1,\ldots,z_{2n}):= (q-q^{-1})^{-n(n-1)} \prod_{1\leq i <j\leq n} (q z_i-q^{-1}z_j)(qz_{n+i}-q^{-1}z_{n+j}).
 \end{align}
Define for two noncrossing matchings $\sigma,\pi$ with $\sigma \nearrow_j \pi$
\begin{equation}
 \label{eq: recursion of wheel poly}
 \Psi_\pi := \D_j(\Psi_{\sigma}) - \sum_{\tau \in e_j^{-1}(\sigma)\setminus \{\sigma, \pi\}} \Psi_\tau.
\end{equation}
Then $\Psi_\pi$ is well-defined for all $\pi \in \NC_n$ and  satisfies
\begin{equation}
 \label{eq: rotation of wheel poly}
 \Psi_{\rho^{-1}(\pi)}(z_1,\ldots, z_{2n})=\Psi_\pi(z_2,\ldots, z_{2n},q^6z_1).
\end{equation}
The set $\{\Psi_\pi, \pi \in \NC_n\}$ is further a $\Q(q)$-basis of $W_n[z]$.
\end{thm}

The noncrossing matchings $\tau$ which appear in the sum of \eqref{eq: recursion of wheel poly} satisfy by Lemma \ref{lem: temperley lieb operator for Young} the relation $\tau < \pi$. Hence we can use \eqref{eq: recursion of wheel poly} to calculate the basis $\Psi_\pi$ of $W_n[z]$ recursively. The vector $\Psi_n=(\Psi_\pi)_{\pi \in \NC_n}$ satisfies \eqref{eq: Rmatrix applied on Psi}. This is true since we can reformulate \eqref{eq: Rmatrix applied on Psi} as
\begin{equation}
 \label{eq: Rmatrix applied on Psi reformulated}
 e_i \Psi_n = \D_i (\Psi_n) -(q+q^{-1})\Psi_n,
\end{equation}
for $1 \leq i \leq 2n-1$. Let $\sigma, \pi \in \NC_n$ with $\sigma \nearrow_i \pi$, then the $\sigma$ component of both sides in \eqref{eq: Rmatrix applied on Psi reformulated} is
\[
 \Psi_\pi-(q+q^{-1})\Psi_\sigma+\sum_{\tau \in e_i^{-1}(\sigma)\setminus \{\sigma,\pi\}} \Psi_\tau = \D_i(\Psi_\sigma) -(q+q^{-1})\Psi_\sigma,
\]
which is exactly \eqref{eq: recursion of wheel poly}.
Since $\Psi_n$ satisfies \eqref{eq: rho applied on Psi} by Theorem \ref{thm: psi basis for wheel poly}, Proposition \ref{prop: qKz and O(tau)} states that $\Psi_n$ is a solution of the qKZ equations and therefore for $\tau=1$ a multiple of the stationary distribution of the inhomogeneous $O(1)$ loop model. By setting $z_1=\ldots =z_{2n}=1$ Theorem \ref{thm: Razumov-Stroganov} implies $\Psi_\pi(1,\ldots,1)|_{\tau=1}=c A_\pi$ for an appropriate constant $c$. Because of $\Psi_{()_n}(1,\ldots,1)|_{\tau=1}=1=A_{()_n}$ by definition, and Theorem \ref{thm: Razumov-Stroganov} we obtain the following theorem.

\begin{thm}
 \label{thm: wheel poly to FPLs}
 Set $q=e^{\frac{2 \pi i}{3}}$ and let $\pi \in \NC_n$, then one has
 \begin{align*}
  \Psi_\pi(z_1,\ldots,z_{2n})&=\ASM(n) \times \hat\Psi_n(t;z_1, \ldots,z_{2n}),\\
  \nonumber \Psi_\pi(1,\ldots,1)&= A_\pi.
 \end{align*}
\end{thm}

\subsection{A new basis for $W_n[z]$}

The following lemma is a direct consequence of the definitions of the $\D_i$'s and $\Psi_{()_n}$.
\begin{lem}
 \label{lem: properties of D-operators}
 Let $n$ be a positive integer, then one has
 \begin{enumerate}
  \item $\D_i \circ \D_i= (q+q^{-1}) \D_i$ for $1 \leq i \leq 2n$,\label{item: squared}
  \item $\D_i \circ \D_j = \D_j \circ \D_i$ for  $1 \leq i,j \leq 2n$ with $|i-j| >1$,\label{item: commute}
  \item $\D_{i+1}\circ \D_i \circ \D_{i+1} +\D_i = \D_i \circ \D_{i+1} \circ \D_i +\D_{i+1}$ for  $1 \leq i \leq 2n$,\label{item: non commute}
  \item $\D_i(\Psi_{()_n})= (q+q^{-1})\Psi_{()_n}$ for $i \notin\{n,2n\}$.\label{item: applied}         
 \end{enumerate}
\end{lem}

In the following we write $\prod_{i=1}^n \D_i$ for the product $\D_1 \circ \ldots \circ \D_n$.
Let $\pi$ be a noncrossing matching with corresponding Young diagram $\lambda(\pi)=(\lambda_1, \ldots, \lambda_l)$, i.\,e., $\lambda_i$ is the number of boxes of $\lambda(\pi)$ in the $i$-th row from top. We define the wheel polynomial $D_\pi$ by the following algorithm. First write in every box of $\lambda(\pi)$ the number of the diagonal the box lies on. The wheel polynomial $D_\pi$ is then constructed recursively by ``reading'' in the Young diagram $\lambda(\pi)$ the rows from top to bottom and in the rows all boxes from left to right and apply $\D_{\textit{number in the box}}$ to the previous wheel polynomial, starting with $\Psi_{()_n}$, which is defined in \eqref{eq: Psi_null}.
 For $\pi$ as in Figure \ref{fig: D pi graphical} we obtain
\[
D_\pi=\left(\D_{n-3} \circ \D_{n-2} \circ \D_{n} \circ \D_{n-1} \circ \D_{n+3} \circ \D_{n+2} \circ \D_{n+1} \circ \D_{n}\right) (\Psi_{()_n}).
\]

Alternatively we can write $D_\pi$ directly as
\begin{equation}
\label{eq: defi of new basis}
 D_\pi=\left( \prod_{i=1}^l\prod_{j=1}^{\lambda_{l+1-i}} \D_{n+(i-l)+(\lambda_{l+1-i}-j)}\right)(\Psi_{()_n}).
\end{equation}

\begin{figure}
 \centering
 \includegraphics[width=0.25\textwidth]{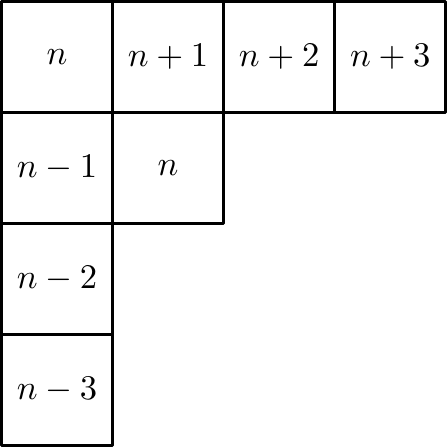}
 \caption{\label{fig: D pi graphical} The numbers indicate the labels of the diagonals the boxes lie on.}
\end{figure}

\begin{thm}
 \label{thm: new basis}
 The set of wheel polynomials $\{D_\pi| \pi \in \NC_n\}$ is a $\Q(q)$-basis of $W_n[z]$. Further $\Psi_\pi$ is for $\pi \in \NC_n$ a linear combination of $D_\tau$'s with $\tau \leq \pi$ and the coefficient of $D_\pi$ is $1$.
\end{thm}
\begin{proof}
 We prove the second statement by induction on the number of boxes of $\lambda(\pi)$. It is by definition true for $()_n$, hence let the number $|\lambda(\pi)|$ be non-zero. Let $\sigma$ be the noncrossing matching such that $\lambda(\sigma)$ is the Young diagram one obtains by deleting the rightmost box in the bottom row of $\lambda(\pi)$, and let $i$ be the integer such that $\sigma \nearrow_i \pi$. Then Theorem \ref{thm: psi basis for wheel poly} states
 \[
  \Psi_\pi=\D_i \Psi_\sigma -\sum_{\tau \in e_i^{-1}(\sigma)\setminus\{\sigma,\pi\}}\Psi_\tau.
 \]
 We use the induction hypothesis to express $\Psi_\tau$ and $\Psi_\sigma$ as sums of $D_{\tau^\prime}$ with $\tau^\prime\leq \tau<\pi$ or $D_{\sigma^\prime}$ with $\sigma^\prime \leq\sigma<\pi$ respectively. The coefficient of $D_\sigma$ in $\Psi_\sigma$ is by the induction hypothesis equals to $1$. 
 Since all $\sigma^\prime\leq \sigma$ satisfy the requirements of Lemma \ref{lem: d-operator applied on new basis}, this lemma implies the statement. By above arguments the set $\{D_\pi|\pi \in \NC_n\}$ is a  $\Q(q)$-generating set for $W_n[z]$ of cardinality $\dim_{\Q(q)}(W_n[z])$, hence it is also a $\Q(q)$-basis.
\end{proof}

The next lemma contains the technicalities which are needed to prove the above theorem.

\begin{lem}
 \label{lem: d-operator applied on new basis}
 Let $1 < i <2n$ and $\sigma \in \NC_n$ such that the number of boxes on the $i$-th diagonal of $\lambda(\sigma)$ is less than the maximal possible number of boxes that can be placed there.
 Then $\D_i (D_\sigma)=D_\pi$ iff there exists a $\pi \in \NC_n$ with $\sigma \nearrow_i \pi$ or otherwise $\D_i (D_\sigma)$ is a $\Q(q)$-linear combination of $D_\tau$'s with $\tau \leq \sigma$.
\end{lem}
\begin{proof}
 We use induction on the number of boxes of $\lambda(\sigma)$. We say that $i$ appears in $\sigma$ if there is a box in $\lambda(\sigma)$ which lies on the $i$-th diagonal.
 
 \begin{enumerate}
  \item Assume that $i$ does not appear in $\sigma$. This implies that $i-1$ can  not appear in $\sigma$. Then there are two cases:
  \begin{enumerate}
   \item   First $i+1$ does not appear in $\sigma$. By Lemma \ref{lem: properties of D-operators} $\D_i$ commutes with all the $\D$-operators appearing in $D_\sigma$. If $i \neq n$ Lemma \ref{lem: properties of D-operators} states $\D_i (\Psi_{()_n})= (q+q^{-1})\Psi_{()_n}$ and hence $\D_i (D_\sigma)=(q+q^{-1})D_\sigma$. The case $i =n$ implies $\sigma=()_n$ and hence $\D_i(D_\sigma)=D_{(()())_{n-2}}$. 
   \item In the second case $i+1$ appears in $\sigma$. Then there is only one box on the $(i+1)$-th diagonal. This box is the leftmost box of the bottom row of $\lambda(\sigma)$. Let $\pi$ be the noncrossing matching whose corresponding Young diagram is obtained by adding a box in a new row in $\lambda(\sigma)$, i.\,e., $\sigma \nearrow_i \pi$. By definition holds $D_\pi=\D_i(D_\sigma)$.
  \end{enumerate}
  
  \item  Let $i$ appear in $\sigma$. We consider the lowest box in the $i$-th diagonal and call it $X$. Let $\sigma^\prime$ be the noncrossing matching of size $n$ whose corresponding Young diagram $\lambda(\sigma^\prime)$ consists of all boxes above and to the left of the box $X$, denote by $\alpha_i$ with $1 \leq i \leq A$ the boxes to the right of $X$ and in the row below but excluding the boxes in the $(i+1)$-th and $(i-1)$-th diagonal and by $\beta_i$ with $1 \leq i \leq B$ the remaining boxes at the bottom. A schematic picture is given in Figure \ref{fig: D pi splitting}.
  Using the previous definitions we can write $D_\sigma$ as 
 \begin{align}
  D_\sigma=\left(\prod_{l=1}^B\D_{\beta_l} \circ \D_{i-1}^b \circ \prod_{l=1}^A\D_{\alpha_l} \circ \D_{i+1}^a \circ \D_i\right) (D_{\sigma^\prime}),
 \end{align}
 
where $a,b$ are $0$ or $1$.

\begin{figure}
 \centering
 \includegraphics[width=0.5\textwidth]{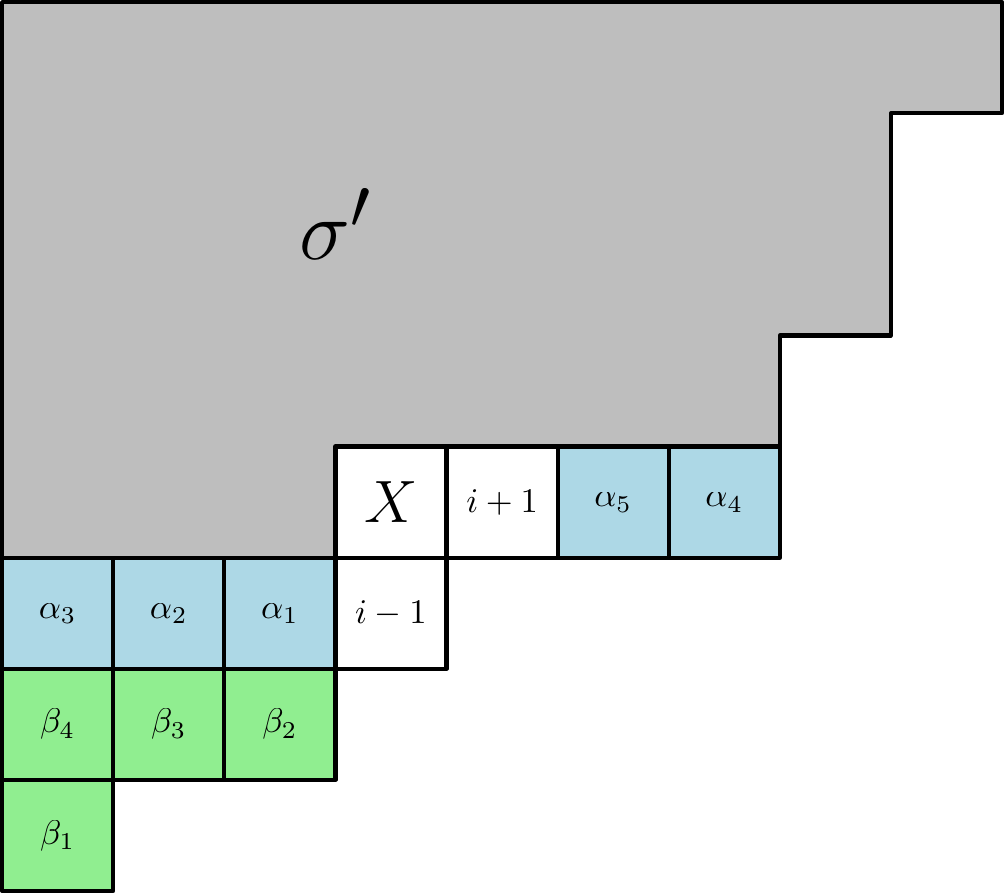}
 \caption{\label{fig: D pi splitting} Schematic representation of $\lambda(\sigma)$ for $\sigma$ as in the second case of the proof of Lemma \ref{lem: d-operator applied on new basis} with $a=b=1$.}
\end{figure}

\begin{enumerate}[(a)]
 \item If $a=b=0$ Lemma \ref{lem: properties of D-operators} (\ref{item: squared},\ref{item: commute}) implies
 \begin{align*}
  \D_i D_\sigma&=
  \D_i \left( \prod_{l=1}^B\D_{\beta_l} \circ \prod_{l=1}^A\D_{\alpha_l} \circ \D_i\right)( D_{\sigma^\prime})\\
  &=\left(\prod_{l=1}^B\D_{\beta_l} \circ \prod_{l=1}^A\D_{\alpha_l} \circ \D_i^2  \right)( D_{\sigma^\prime})\\
  &=\left(\prod_{l=1}^B\D_{\beta_l} \circ \prod_{l=1}^A\D_{\alpha_l} \circ ((q+q^{-1})\D_i)\right)( D_{\sigma^\prime}) =(q+q^{-1})D_\sigma.
 \end{align*}

 \item For $a=b=1$, the operator $\D_i$ commutes with all $\D_{\beta_l}$. As Figure \ref{fig: D pi splitting} shows and by the assumptions on $\sigma$ there exists a noncrossing matching $\pi$ with $\sigma \nearrow_i \pi$. Hence one has 
 \[
  \D_i(D_\sigma) =
 \left(\prod_{l=1}^B\D_{\beta_l}\circ  \D_i \circ \D_{i-1} \circ \prod_{l=1}^A\D_{\alpha_l} \circ \D_{i+1} \circ \D_i\right) (D_{\sigma^\prime})
  =D_\pi.
 \] 
 \item For $a=1,b=0$ we obtain by Lemma \ref{lem: properties of D-operators} (\ref{item: commute},\ref{item: non commute})
 \begin{align*}
 \D_i(D_\sigma)&= \left(\prod_{l=1}^B \D_{\beta_l} \circ \prod_{l=1}^A \D_{\alpha_l} \circ  \D_i \circ \D_{i+1} \circ \D_i\right)( D_{\sigma^\prime}) \\
 &=\left(\prod_{l=1}^B\D_{\beta_l} \circ \prod_{l=1}^A \D_{\alpha_l} \right) \left( (\D_{i+1} \circ \D_i \circ \D_{i+1} +\D_i -\D_{i+1})(D_{\sigma^\prime})\right).
 \end{align*}
 By the induction hypothesis $\left( (\D_{i+1} \circ \D_i \circ \D_{i+1} +\D_i -\D_{i+1})(D_{\sigma^\prime})\right)$ is a linear combination of $D_\tau$'s with $\tau \leq \hat\sigma$ where $\hat\sigma$ is $\sigma^\prime$ with a box added on the $i$-th and $i+1$-th diagonal. Using again the induction hypothesis for the $D_\tau$'s with $\tau \leq \hat\sigma$ proofs the claim.
 \item Let $a=0,b=1$ and let $\hat{\sigma}$ be the noncrossing matching whose Young diagram consists of $\lambda(\sigma^\prime)$ and the boxes labelled with $\alpha_i$ for $1 \leq i \leq A$. Lemma \ref{lem: properties of D-operators} (\ref{item: commute},\ref{item: non commute}) implies
  \begin{align*}
 \D_i(D_\sigma)&= \left(\prod_{l=1}^B \D_{\beta_l} \circ  \D_i \circ \D_{i-1} \circ \D_i \circ \prod_{l=1}^A \D_{\alpha_l} \right)( D_{\sigma^\prime}) \\
 &=\left(\prod_{l=1}^B\D_{\beta_l} \circ (\D_{i-1} \circ \D_i \circ \D_{i-1} +\D_i -\D_{i-1}) \right) (D_{\hat\sigma}).
 \end{align*}
We finish the proof by using the induction hypothesis analogously to the above case.
 \qedhere
\end{enumerate}
\end{enumerate}
\end{proof}

Let $\pi \in \NC_n$ be a noncrossing matching given by $\pi=\pi_1\pi_2$ where $\pi_i$ is a noncrossing matching of size $n_i$ for $i=1,2$. We want to generalise $D_\pi$ and Theorem \ref{thm: new basis} in the sense that we can write $\Psi_{\pi}=\Psi_{\pi_1\pi_2}$ as a linear combination of $D_{\tau_1,\tau_2}$ with $\tau_i\leq \pi_i$ for $i=1,2$. This will not be possible for $\Psi_{\pi}$ but for $\Psi_{\rho^{n_2}\pi}$. Let the Young diagram corresponding to $\pi_2$ be given as $\lambda(\pi_2)=(\lambda_1,\ldots, \lambda_l)$.
The wheel polynomial $D_{\pi_1,\pi_2}$ is then defined by the following algorithm. First we write in every box of $\lambda(\pi_2)$ the number of the diagonal the box lies on. The wheel polynomial $D_{\pi_1\pi_2}$ is then constructed recursively by ``reading'' in the Young diagram $\lambda(\pi_2)$ the rows from top to bottom and in the rows all boxes from left to right and apply $\D_{\textit{number in the box}-n}$ to the previous wheel polynomial, starting with $D_{(\pi_1)_{n_2}}$, which is defined in \eqref{eq: defi of new basis}.
Remember that we have extended the definition of $\D_k$ to all integers via $\D_k=\D_{k+2n}$.
We can express $D_{\pi_1,\pi_2}$ also by the following formula
\[
 D_{\pi_1,\pi_2}:=\left( \prod_{i=1}^l\prod_{j=1}^{\lambda_{l+1-i}} \D_{(i-l)+(\lambda_{l+1-i}-j)}\right)(D_{(\pi_1)_{n_2}}).
\]

 For $\pi_2$ as in Figure \ref{fig: D pi graphical} we obtain
\[
D_{\pi_1,\pi_2}=\left(\D_{-3} \circ \D_{-2} \circ \D_{0} \circ \D_{-1} \circ \D_{3} \circ \D_{2} \circ \D_{1} \circ \D_{0}\right) (D_{(\pi_1)_{n_2}}).
\]

\begin{thm}
 \label{thm: new basis for splitted matchings}
 Let $\pi=\pi_1\pi_2, \pi_1,\pi_2$ be noncrossing matching of size $n,n_1$ or $n_2$ respectively and set $q=e^{\frac{2\pi i}{3}}$. The wheel polynomial 
 \[
 \Psi_{\rho^{n_2}(\pi_1\pi_2)}(z_1,\ldots,z_{2n})= \Psi_{\pi_1\pi_2}(z_{2n+1-n_2},\ldots,z_{2n},z_1,\ldots, z_{2n-n_2})
\] 
  can be expressed as a linear combination of $D_{\tau_1,\tau_2}$'s where $\tau_i \leq \pi_i$ and the coefficient of $D_{\pi_1,\pi_2}$ is $1$.
\end{thm}

\begin{proof}
We calculate $\Psi_{\rho^{n_2}(\pi_1\pi_2)}$ in three steps:
\begin{enumerate}
\item $\Psi_{(\pi_1)_{n_2}}$ is by Theorem \ref{thm: new basis} a linear combination of $D_{(\tau_1)_{n_2}}$'s with $\tau_1 \leq \pi_1$ and the coefficient of $D_{(\pi_1)_{n_2}}$ is $1$.

\item Theorem \ref{thm: psi basis for wheel poly} implies 
\[
 \Psi_{\pi_1()_{n_2}} = \Psi_{\rho^{-n_2}((\pi_1)_{n_2})} = \Psi_{(\pi_1)_{n_2}}(z_{n_2+1},\ldots, z_{2n},z_1,\ldots,z_{n_2}).
\]

\item Use the recursion \eqref{eq: recursion of wheel poly} of Theorem \ref{thm: psi basis for wheel poly} to obtain $\Psi_{\pi_1\pi_2}$ starting from $\Psi_{\pi_1()_{n_2}}$. By Lemma \ref{lem: temperley lieb operator for Young} the $\tau$ appearing in the sum in \eqref{eq: recursion of wheel poly} are of the form $\pi_1\tau_2$ with $\tau_2 \leq \pi_2$.
\end{enumerate}

The algorithm for calculating $\Psi_{(\pi_2)_{n_1}}$ and the third step of calculating $\Psi_{\pi_1\pi_2}$ differ by the initial condition -- in the first case $\Psi_{()_n}$, in the second $\Psi_{\pi_1()_{n_2}}$ -- and each $\D_i$ of the first algorithm is replaced by $\D_{i+n_2}$.
Hence we can use Theorem \ref{thm: new basis} to express $\Psi_{\pi_1\pi_2}$ as a linear combination of $\hat D_{\tau_2}$ with $\tau_2 \leq \pi_2$, where $\hat D_{(\tau_2)_{n_1}}$ is obtained by taking $D_{(\tau_2)_{n_1}}$ and changing every $\D_i$ to a $\D_{i+n_2}$ and $\Psi_{()_n}$ is replaced by $\Psi_{\pi_1()_{n_2}}$. Together with the first two parts of the algorithm this implies that $\Psi_{\rho^{n_2}(\pi_1\pi_2)}$ is a linear combination of $D_{\tau_1,\tau_2}$'s with $\tau_i \leq \pi_i$ and the coefficient of $D_{\pi_1,\pi_2}$ is $1$.
\end{proof}

\begin{rem}
Let $\Psi_{\pi_i}=\sum_{\tau_i \leq \pi_i}\alpha_{\tau_i}D_{\tau_i}$ for $i=1,2$. The above proof implies
\begin{align*}
\Psi_{\rho^{n_2}(\pi_1\pi_2)}=\sum_{\tau_1 \leq \pi_1,\tau_2 \leq \pi_2}\alpha_{\tau_1}\alpha_{\tau_2}D_{\tau_1,\tau_2}.
\end{align*}
Hence gaining knowledge about 
\[
 A_{\pi_1\pi_2}=\Psi_{\pi_1\pi_2}|_{z_1=\ldots=z_{2n}=1,q^3=1}=\Psi_{\rho^{n_2}(\pi_1\pi_2)}|_{z_1=\ldots=z_{2n}=1,q^3=1}
\]
could be achieved by understanding the coefficients $\alpha_{\tau_i}$ and the behaviour of $D_{\tau_1,\tau_2}$ for $\tau_i \leq \pi_i$. However this seems to be very difficult.
\end{rem}

\section{Fully packed loops with a set of nested arches}

\noindent In order to prove Theorem \ref{thm: main thm} we will need to calculate $D_{\pi_1,\pi_2}$ at $z_1= \ldots =z_{2(n_1+n_2)}=1$ for two noncrossing matchings $\pi_1,\pi_2$.
The following notations will simplify this task. We define
\begin{align*}
 f(i,j)&:=\frac{qz_i-q^{-1}z_j}{q-q^{-1}}, \quad g(i):=\frac{q-q^{-1}z_i}{q-q^{-1}},\quad
 h(i):= \frac{qz_i-q^{-1}}{q-q^{-1}},
\end{align*}
for $1 \leq i\neq j \leq 2n$.
Using this notations we obtain
\[
 \Psi_{()_n}=\prod_{1\leq i<j \leq n}f(i,j)f(n+i,n+j).
\]
 One verifies the following lemma by simple calculation.

\begin{lem}
\label{lem: operatoren explicit}
For $1 \leq i,j,k \leq 2n$ and $i \neq j$ one has
 \begin{enumerate}
 \item
 \label{item: alpha}
  $\D_k(f(i,j))=
 \begin{cases}
(q+q^{-1}) f(k,k+1)  \quad &(i,j)=(k,k+1),\\
 -(q+q^{-1})f(k,k+1) &(i,j)=(k+1,k),\\
 qf(k,k+1) & i=k; \; j\neq k+1,\\
 -qf(k,k+1) & i=k+1; \; j\neq k,\\
 -q^{-1}f(k,k+1) & j=k; \; i\neq k+1,\\
  q^{-1}f(k,k+1) & j=k+1; \; i\neq k,\\
 0 & \{i,j\}\cap \{k,k+1\}= \emptyset,
 \end{cases}$
 
 \item 
 \label{item: beta}
 $\D_k(g(i))=
\begin{cases}
-q^{-1}f(k,k+1) & i=k,\\
q^{-1}f(k,k+1) \quad & i=k+1,\\
0 & \text{otherwise},
 \end{cases}$
 
 \item 
  \label{item: gamma}
  $\D_k(h(i))=
\begin{cases}
qf(k,k+1) \quad & i=k,\\
-qf(k,k+1) & i=k+1,\\
0 & \text{otherwise}.
 \end{cases}$
 
\item
\label{item: potenzen} Let $m$ be a positive integer, then the following holds
\begin{align*}
\D_k(f(i,j)^m)&=\D_k(f(i,j))\sum_{l=0}^{m-1}f(i,j)^l \S_k(f(i,j)^{m-1-l}),\\
\D_k(g(i)^m)&=\D_k(g(i))\sum_{l=0}^{m-1}g(i)^l \S_k(g(i)^{m-1-l}),\\
\D_k(h(i)^m)&=\D_k(h(i))\sum_{l=0}^{m-1}h(i)^l \S_k (h(i)^{m-1-l}).
\end{align*}
\end{enumerate}
\end{lem}

We further introduce the abbreviation
\begin{align*}
 P(\alpha_{i,j}|\beta_i|\gamma_i):=\prod_{1\leq i \neq j\leq 2n}f(i,j)^{\alpha_{i,j}}\prod_{i=1}^{2n}
g(i)^{\beta_i}h(i)^{\gamma_i},
\end{align*}
where $\alpha_{i,j},\beta_i,\gamma_i$ are nonnegative integers for $1 \leq i \neq j \leq 2n$.
Our goal is to obtain a useful expression for $\D_{i_1}\circ \cdots \circ \D_{i_m} ( P(\alpha_{i,j}|\beta_i|\gamma_i))|_{z_1= \ldots = z_{2n}=1}$ for special values of $\alpha_{i,j},\beta_i$ and $\gamma_i$. By using the previous lemma it is very easy to see that $\D_{i_1}\circ \cdots \circ \D_{i_m}(P(\alpha_{i,j}|\beta_i|\gamma_i))$ is a sum of products of the form $P(\alpha^\prime_{i,j}|\beta^\prime_i|\gamma^\prime_i)$. The explicit form of this sum is easy to understand when only one $\D$-operator is applied but gets very complicated for more. However it turns out that $\D_{i_1}\circ \cdots \circ \D_{i_m} ( P(\alpha_{i,j}|\beta_i|\gamma_i))|_{z_1= \ldots = z_{2n}=1}$ is a polynomial in $\alpha_{i,j},\beta_i$ and $\gamma_i$, which is stated in Lemma \ref{lem: evaluation of P}. The next example hints at the basic idea behind this fact.

\begin{ex}
\label{ex: polynom ausrechnen}
Let $P=P(\alpha_{i,j}|\beta_i|\gamma_i)$ and $n=1$. We calculate $\D_1(P)_{z_1=z_2=1}$ explicitly. By using Lemma \ref{lem: S,D operator} and Lemma \ref{lem: operatoren explicit} we obtain for $\D_1(P)$ the expression.
\begin{align*}
\D_1(P)= &\D_1 \left( f(1,2)^{\alpha_{1,2}}f(2,1)^{\alpha_{2,1}}g(1)^{\beta_1}g(2)^{\beta_2}h(1)^{\gamma_1}h(2)^{\gamma_2}\right)\\
=&(q+q^{-1})\sum_{t=0}^{\alpha_{1,2}-1}f(1,2)^{\alpha_{1,2}+\alpha_{2,1}-t}f(2,1)^{t}g(1)^{\beta_2}g(2)^{\beta_1}h(1)^{\gamma_2}h(2)^{\gamma_1}+\\
&-(q+q^{-1})\sum_{t=0}^{\alpha_{2,1}-1}f(1,2)^{\alpha_{1,2}+\alpha_{2,1}-t}f(2,1)^{t}g(1)^{\beta_2}g(2)^{\beta_1}h(1)^{\gamma_2}h(2)^{\gamma_1}+\\
&-q^{-1}\sum_{t=0}^{\beta_1-1}f(1,2)^{\alpha_{1,2}+1}f(2,1)^{\alpha_{2,1}}g(1)^{\beta_1+\beta_2-t-1}g(2)^{t}h(1)^{\gamma_2}h(2)^{\gamma_1}+\\
&+q^{-1}\sum_{t=0}^{\beta_2-1}f(1,2)^{\alpha_{1,2}+1}f(2,1)^{\alpha_{2,1}}g(1)^{\beta_1+\beta_2-t-1}g(2)^{t}h(1)^{\gamma_2}h(2)^{\gamma_1}+\\
&+q\sum_{t=0}^{\gamma_1-1}f(1,2)^{\alpha_{1,2}+1}f(2,1)^{\alpha_{2,1}}g(1)^{\beta_1}g(2)^{\beta_2}h(1)^{\gamma_1+\gamma_2-t-1}h(2)^{t}+\\
&-q\sum_{t=0}^{\gamma_2-1}f(1,2)^{\alpha_{1,2}+1}f(2,1)^{\alpha_{2,1}}g(1)^{\beta_1}g(2)^{\beta_2}h(1)^{\gamma_1+\gamma_2-t-1}h(2)^{t}.
\end{align*}
By evaluating this at $z_1=z_2=1$ we obtain:
\[
\D_1(P)|_{z_1=z_2=1} = (q+q^{-1})(\alpha_{1,2}-\alpha_{2,1}) +q^{-1}(\beta_2-\beta_1)+q(\gamma_1-\gamma_2),
\]
which is a polynomial in the $\alpha_{i,j},\beta_i,\gamma_i$.\\
\end{ex}

The proof of Theorem \ref{thm: main thm} is achieved by using two main ingredients. First Theorem \ref{thm: new basis for splitted matchings} allows us to express the wheel polynomial $\Psi_{(\pi_1)_m\pi_2}$ in a suitable basis and second Lemma \ref{lem: evaluation of P} tells us what we have to expect when evaluating the basis at $z_1=\ldots=z_{2N}=1$. 

\begin{proof}[Proof of Theorem \ref{thm: main thm}]

In the following we show that the number $A_{(\pi_1)_m \pi_2}$ of FPLs with link pattern $(\pi_1)_m \pi_2$ is a polynomial in $m$. Together with \cite[Theorem $6.7$]{on_the_number_of_FPL}, which states that $A_{(\pi_1)_m \pi_2}$ is a polynomial in $m$ with requested degree and leading coefficient for large values of $m$, this proves Theorem \ref{thm: main thm}.\\

Set $N=m+n_1+n_2$ and $q=e^{\frac{2 \pi i}{3}}$. By Theorem \ref{thm: wheel poly to FPLs}, Theorem \ref{thm: Razumov-Stroganov} and Theorem \ref{thm: psi basis for wheel poly} one has
\[
A_{(\pi_1)_m\pi_2}=\Psi_{(\pi_1)_m\pi_2}|_{z_1=\ldots=z_{2N}=1}=\Psi_{\rho^{n_2}((\pi_1)_m\pi_2)}|_{z_1=\ldots=z_{2N}=1}.
\]
Theorem \ref{thm: new basis for splitted matchings} implies that $\Psi_{\rho^{n_2}((\pi_1)_m\pi_2)}$ is a linear combination of $D_{(\tau_1)_m,\tau_2}$ with $\tau_i \leq \pi_i$ for $i=1,2$. By definition $D_{(\tau_1)_m,\tau_2}$ is of the form $\prod_{j=1}^k\D_{i_j} (\Psi_{()_N})$ with $k \leq|\lambda(\pi_1)|+|\lambda(\pi_2)|$ and $i_j \in \{1,\ldots ,n_2-2 ,N-n_1+2,\ldots,N+n_1-2,2N-n_2+2,\ldots,2N\}$ for $1 \leq j \leq k$. The operator $\D_{i_j}$ acts for $1 \leq j \leq k$ trivially on $z_i$ with $i \in I:=\{n_2+1,\ldots, N-n_1,N+n_1+1,\ldots,2N-n_2\}$. Hence one has
\[
 \left.\left(\prod_{j=1}^k \D_{i_j} (\Psi_{()_N})\right)\right|_{z_1=\ldots=z_{2N}=1}=
 \left.\left(\prod_{j=1}^k \D_{i_j} (\Psi_{()_N| \forall i \in I: z_i=1})\right)\right|_{\forall i \in \{1,\ldots,2N\}\setminus I: z_i=1}.
\]

The polynomial $\left.\Psi_{()_N}\right|_{z_i=1 \forall i \in I}$ is a polynomial in the $2(n_1+n_2)$ variables $z_i$, where $i$ is an element of  $\{1,\ldots,2N\}\setminus I$. For simplicity we substitute these remaining variables with $z_1, \ldots, z_{2(n_1+n_2)}$ whereby we keep the same order on the indices.
Hence $\left.\Psi_{()_N}\right|_{z_i=1 \forall i \in I}$ can be written in the form $P=P(\alpha_{i,j}|\beta_i|\gamma_i)$ with 

 \begin{align*}
  \alpha_{i,j}&=
  \begin{cases}
   1 \quad &  i<j \textit{ and } \left(j\leq n_1+n_2 \text{ or } i >n_1+n_2\right),\\
   0 &\text{otherwise},
  \end{cases}\\
 \beta_i&=
 \begin{cases}
  m \quad &i \in \{n_2+1, \ldots, n_1+n_2,2n_1+n_2+1, \ldots, 2(n_1+n_2)\},\\
  0 &\text{otherwise},
  \end{cases}\\
 \gamma_i&=
 \begin{cases}
 m \quad & i \in \{1, \ldots, n_2,n_1+n_2+1, \ldots, 2n_1+n_2\},\\
 0 &\text{otherwise}, 
 \end{cases}
 \end{align*}
 whereas all the $z_i$ in $f(i,j),g(i)$ and $h(i)$ are replaced by $\hat{z}_i$.
Lemma \ref{lem: evaluation of P} implies that $\prod_{j=1}^k\D_{i_j}(P)$ is a polynomial in $m$ of degree at most $k\leq|\lambda(\pi_1)|+|\lambda(\pi_2)|$ which proves the statement.
\end{proof}

We conclude the proof of Theorem \ref{thm: main thm} by the following Lemma.

\begin{lem}
\label{lem: evaluation of P}
 Let $P=P(\alpha_{i,j}|\beta_i|\gamma_i)$, $m$ an integer and $i_1, \ldots, i_m \in \{1, \ldots, 2n\}$. There exists a polynomial $Q \in \mathbb{Q}(q)[y_1,\ldots,y_{2n(2n+1)}]$ with total degree at most $m$ such that 
 \[
 \left. \D_{i_1}\circ\cdots\circ \D_{i_m} (P)\right|_{z_1=\ldots=z_{2n}=1}
=Q((\alpha_{i,j}),(\beta_i),(\gamma_i)).
\]
\end{lem}

\begin{proof} We prove the theorem by induction on $m$. The statement is trivial for $m=0$, hence let $m>0$ and set $k=i_m$. We can express $\D_k(P)$ as
\begin{equation}
\label{eq: splitting of Dk}
\D_{k}P= \sum_{s \in S}a_s P_s,
\end{equation}
for a finite set $S$ of indices, $a_s \in \{\pm q, \pm q^{-1}, \pm (q+q^{-1}) \}$ and $P_s=P(\alpha_{s;i,j}|\beta_{s;i}|\gamma_{s;i})$ for all $s \in S$.
Indeed we can use iteratively the product rule for the operator $\D_k$, stated in Lemma \ref{lem: S,D operator}, to split $\D_k(P)$ into a sum. Since this splitting depends on the order of the factors, we fix it to be
\begin{align*}
 P=\prod_{i=1}^{2n}\prod_{\substack{j=1,\\ j \neq i}}^{2n} f(i,j)^{\alpha_{i,j}} \prod_{i=1}^{2n}g(i)^{\beta_i} \prod_{i=1}^{2n}h(i)^{\gamma_i}.
\end{align*}
 Lemma \ref{lem: operatoren explicit} implies that every summand is of the form $P_s=P(\alpha_{s;i,j}|\beta_{s;i}|\gamma_{s;i})$ and the coefficients $a_s$ are as stated above, which verifies \eqref{eq: splitting of Dk}.\\
 
 We express $\D_k(P)$ more explicitly by using the above defined ordering of the factors and Lemma \ref{lem: S,D operator}
\begin{subequations}
\begin{align}
\D_{k}(P) &=
\D_{k} \left( \prod_{1\leq i\neq j\leq 2n}f(i,j)^{\alpha_{i,j}} 
\prod_{i=1}^{2n} g(i)^{\beta_i}h(i)^{\gamma_i} \right)\nonumber \\
&= \sum_{1 \leq i \neq j \leq 2n} \prod_{\substack{1 \leq i^\prime \neq j^\prime\leq 2n\\
(i^\prime<i) \lor (i^\prime =i,j^\prime < j)}}f(i^\prime,j^\prime)^{\alpha_{i^\prime,j^\prime}}\times
\D_k(f(i,j)^{\alpha_{i,j}}) \nonumber \\
\label{eq: calculating D_i_k A}
&\times \S_k \left(\prod_{\substack{1 \leq i^\prime \neq j^\prime\leq 2n\\
(i^\prime>i) \lor (i^\prime =i, j^\prime > j)}}f(i^\prime,j^\prime)^{\alpha_{i^\prime,j^\prime}} \prod_{i^\prime=1}^{2n} g(i^\prime)^{\beta_{i^\prime}}h(i^\prime)^{\gamma_{i^\prime}}\right)\\
&+\sum_{i=1}^{2n} \prod_{1 \leq i^\prime \neq j^\prime\leq 2n}f(i^\prime,j^\prime)^{\alpha_{i^\prime,j^\prime}}
\prod_{i^\prime=1}^{i-1}g(i^\prime)^{\beta_{i^\prime}} \times \D_k(g(i)^{\beta_{i}})\nonumber \\
\label{eq: calculating D_i_k B}
&\times \S_k\left(\prod_{i^\prime=i+1}^{2n}g(i^\prime)^{\beta_{i^\prime}} \prod_{i^\prime=1}^{2n}h(i^\prime)^{\gamma_{i^\prime}} \right)\\
&+\sum_{i=1}^{2n} \prod_{1 \leq i^\prime \neq j^\prime\leq 2n}f(i^\prime,j^\prime)^{\alpha_{i^\prime,j^\prime}}
\prod_{i^\prime=1}^{2n}g(i^\prime)^{\beta_{i^\prime}}
\prod_{i^\prime=1}^{i-1}h(i^\prime)^{\gamma_{i^\prime}} \times \D_k(h(i)^{\gamma_{i}})\nonumber\\
\label{eq: calculating D_i_k C}
&\times \S_k\left(\prod_{i^\prime=i+1}^{2n}h(i^\prime)^{\gamma_{i^\prime}} \right).
\end{align}
\end{subequations}

Using Lemma \ref{lem: operatoren explicit} we split every summand in \eqref{eq: calculating D_i_k A} up into a sum of $P_s$ with $s \in S$ and say that these $P_s$ originate from this very summand. We define $A_{i,j}$ for $1 \leq i \neq j \leq 2n$ to be the set consisting of all $s\in S$ such that $P_s$ originates from the summand in \eqref{eq: calculating D_i_k A} with control variables $i,j$. analogously we define for $1 \leq i \leq 2n$ the sets $B_i$ and $C_i$ to consist of all $s \in S$ such that $P_s$ originates from the summand with control variable $i$ in \eqref{eq: calculating D_i_k B} or \eqref{eq: calculating D_i_k C} respectively. Hence we can write the set $S$ as the disjoint union
\[
S= \left(\bigcup_{1 \leq i\neq j \leq 2n}A_{i,j}\right) \cup \left(\bigcup_{1 \leq i \leq 2n}B_i\right) \cup \left(\bigcup_{1 \leq i \leq 2n}C_i\right).
\]

Lemma \ref{lem: operatoren explicit} implies $\D_k(f(i,j))=0$ for $\{i,j\}\cap \{{k},{k}+1\}= \emptyset$ and $\D_k(g(i))=\D_k(h(i))=0$  for $i \notin \{k,k+1\}$. Therefore the sets $A_{i,j},B_i,C_i$ are empty in these cases.\\

Let $1\leq i \neq j \leq 2n$ be fixed with $\{i,j\}\cap \{{k},{k}+1\} \neq \emptyset$ and let $\sigma \in \mathfrak{S}_{2n}$ be the permutation $\sigma=({k},{k}+1)$. Set $\Lambda_{i,j}=\{(i^\prime,j^\prime): 1 \leq i^\prime \neq j^\prime\leq 2n, (i^\prime<i) \lor (i^\prime =i,j^\prime < j)\}$. The definition of $A_{i,j}$ and Lemma \ref{lem: operatoren explicit} imply for all $(i^\prime,j^\prime) \notin \{(i,j),(\sigma(i),\sigma(j)),(k,k+1)\}$ and all $s \in A_{i,j}$:

\begin{align*}
\alpha_{s;i^\prime,j^\prime}&=
\begin{cases}
\alpha_{i^\prime,j^\prime} \quad &\{i^\prime,j^\prime\}\cap\{k,k+1\}= \emptyset \text{ or } \left((i^\prime,j^\prime), (\sigma(i^\prime),\sigma(j^\prime)) \in \Lambda_{i,j}\right),\\
\alpha_{i^\prime,j^\prime}+\alpha_{\sigma(i^\prime),\sigma(j^\prime)} &\{i^\prime,j^\prime\}\cap\{k,k+1\}\neq \emptyset, (i^\prime,j^\prime) \in \Lambda_{i,j}, (\sigma(i^\prime),\sigma(j^\prime)) \notin  \Lambda_{i,j},\\
0 &\{i^\prime,j^\prime\}\cap\{k,k+1\}\neq \emptyset, (i^\prime,j^\prime) \notin \Lambda_{i,j}, (\sigma(i^\prime),\sigma(j^\prime)) \in  \Lambda_{i,j},\\
\alpha_{\sigma(i^\prime),\sigma(j^\prime)} &\{i^\prime,j^\prime\}\cap\{k,k+1\}\neq \emptyset, (i^\prime,j^\prime),(\sigma(i^\prime),\sigma(j^\prime)) \notin \Lambda_{i,j}.
\end{cases}
\end{align*}
If $(k,k+1) \notin \{(i,j),(\sigma(i),\sigma(j))\}$, the parameter $\alpha_{s;k,k+1}$ is given as the adequate value of the above case analysis added by $1$. Further we obtain $\beta_{s;i^\prime} = \beta_{\sigma(i^\prime)}$  and $\gamma_{s;i^\prime}= \gamma_{\sigma(i^\prime)}$ for all $1 \leq i^\prime \leq 2n$  and $s \in A_{i,j}$. By Lemma \ref{lem: operatoren explicit} the constant $a_s$ is for all $s \in A_{i,j}$ determined by the corresponding constant of $\D_k(f(i,j))$  and hence not depending on $s$. The last statement of Lemma \ref{lem: operatoren explicit} implies that we can list the elements of $A_{i,j} = \{s_1, \ldots, s_{\alpha_{i,j}}\}$ such that we have the following description for the remaining parameters $\alpha_{s;i,j}$  and $\alpha_{s;\sigma(i),\sigma(j)}$ :\\
\begin{align*}
\alpha_{s_t;i,j}&=
\begin{cases}
\alpha_{i,j}+\alpha_{j,i}+1-t \quad & i=k,j=k+1,\\
\alpha_{i,j}-t & i=k+1,j=k,\\
\alpha_{i,j}+\alpha_{\sigma(i),\sigma(j)}-t & \{i,j\}\cap\{k,k+1\}=\{k\},\\
\alpha_{i,j}-t & \{i,j\}\cap\{k,k+1\}=\{k+1\},
\end{cases}\\
\alpha_{s_t;\sigma(i),\sigma(j)}&=
\begin{cases}
\alpha_{i,j}+\alpha_{j,i}-\alpha_{s_t;i,j} & \{i,j\}=\{k,k+1\},\\
\alpha_{i,j}+\alpha_{\sigma(i),\sigma(j)}-\alpha_{s_t;i,j}-1 & \text{otherwise},
\end{cases}
\end{align*}
with $1 \leq t \leq \alpha_{i,j}$. If $k=2n$ the first two and last two cases in the description of $\alpha_{s_t;i,j}$ switch places, which is due to the fact that we identify $k+1$ with $1$ for $k=2n$.\\

There exists an analogue description for the sets $B_i, C_i$ and $i \in \{{k},{k}+1\}$  as above, whereas the only parameters that change are given in the case of $B_i$ by
\[
 \beta_{s_t;k}=\beta_{{k}}+\beta_{{k}+1}-t, \qquad \beta_{s_t;k+1}= t-1,
\]
with $1 \leq t \leq \beta_i$ and in the case of $C_i$ by 
\[
 \gamma_{s_t;k}=\gamma_{{k}}+\gamma_{{k}+1}-t, \qquad \gamma_{s_t;k+1}= t-1,
\]
with $1 \leq t \leq \gamma_i$. For $k=2n$ the description of $\beta_{s_t;k},\beta_{s_t;k+1}$ and $\gamma_{s_t;k},\gamma_{s_t;k}$ are interchanged.\\

We know by induction that 
$\D_{i_1}\circ\cdots \circ \D_{i_{m-1}} \left(P(a_{i,j}|b_{i}|c_{i}) \right)|_{z_1=\ldots=z_{2n}=1}$ is a polynomial $Q^\prime$ of degree at most $m-1$ in $ (a_{i,j}),(b_{i})$ and $(c_{i})$. Since the operators $\D_{i}$ are linear we can write
\begin{multline}
 \left.\D_{i_1}\circ\cdots\circ \D_{i_m} (P)\right|_{z_1=\ldots=z_{2n}=1} \\
 =\left.\D_{i_1}\circ\cdots\circ \D_{i_{m-1}}\left( \sum_{s \in S} a_s P(\alpha_{s;i,j}|\beta_{s;i}|\gamma_{s;i}) \right)\right|_{z_1=\ldots=z_{2n}=1}\\
= \sum_{s \in S} a_s \D_{i_1}\circ\cdots\circ \D_{i_{m-1}} \left(P(\alpha_{s;i,j}|\beta_{s;i}|\gamma_{s;i})\right)|_{z_1=\ldots=z_{2n}=1} \\
=\sum_{s \in S}a_s Q((\alpha_{s;i,j}),(\beta_{s;i}),(\gamma_{s;i})).
\end{multline}
The  description above implies that if we restrict ourselves to $s\in A_{i,j}$, $s \in B_i$ or $s \in C_i$ respectively, $a_s$ is independent of $s$, the parameters $\alpha_{s;i^\prime,j^\prime},\beta_{s;i^\prime},\gamma_{s;i^\prime}$ are constant for $(i^\prime,j^\prime)\neq (i,j),(\sigma(i),\sigma(j))$ or $i^\prime \neq k,k+1$ respectively and otherwise depending linearly on a parameter $t$ which runs from $1$ up to the cardinality of the set $A_{i,j}$, $B_i$ or $C_i$ respectively.
The fact, that for a polynomial $p(t)$ of degree $d$ the sum $\sum_{x \leq t \leq y}p(t)$ is a polynomial in $x$ and $y$ of degree at most $d+1$, together with the previous statement imply that the sum
\begin{align*}
&\sum_{s \in A_{i,j}}a_s Q((\alpha_{s;i,j}),(\beta_{s;i}),(\gamma_{s;i})),
\end{align*}
and the analogous sums for $s \in B_i$ or $s \in C_i$ respectively are polynomials in $ (\alpha_{i,j}),(\beta_{i}),(\gamma_{i})$ of degree at most $m$ for all $1 \leq i\neq j \leq 2n$. Therefore
\begin{align*}
\D_{i_1}\circ\cdots \circ \D_{i_{m}}(P)|_{z_1=\ldots=z_{2n}=1}&=\sum_{s \in S} a_s Q((\alpha_{s;i,j}),(\beta_{s;i}),(\gamma_{s;i})),
\end{align*}
is a polynomial in $ (\alpha_{i,j}),(\beta_{i}),(\gamma_{i})$ of degree at most $m$.
\end{proof}

\end{document}